\documentclass[12pt]{amsart}
\usepackage[english]{babel}
\usepackage{paralist,url,verbatim, anysize}
\usepackage{amscd}
\usepackage[all,cmtip]{xy} 
\usepackage{amsmath,amsthm,amssymb,amsfonts}
\usepackage[bookmarksopen=true]{hyperref}
\usepackage[usenames, dvipsnames]{color}
\usepackage{graphicx}
\usepackage{hyperref}
\usepackage{euler, amsfonts, amssymb, latexsym, epsfig,epic}
\usepackage{stmaryrd}

\makeatletter
\def\proof{\@ifnextchar[{\@oproof}{\@nproof}}
\def\@oproof[#1][#2]{\trivlist\item[\hskip\labelsep
\textit{#2 Proof of\ #1.}~]\ignorespaces}
\def\@nproof{\trivlist\item[\hskip\labelsep\textit{Proof.}~]\ignorespaces}

\makeatother

\setbox0=\hbox{$+$}
\newdimen\plusheight
\plusheight=\ht0
\def\+{\;\lower\plusheight\hbox{$+$}\;}

\setbox0=\hbox{$-$}
\newdimen\minusheight
\minusheight=\ht0
\def\-{\;\lower\minusheight\hbox{$-$}\;}

\setbox0=\hbox{$\cdots$}
\newdimen\cdotsheight
\cdotsheight=\plusheight
\def\cds{\lower\cdotsheight\hbox{$\cdots$}}

\DeclareMathOperator{\Hom}{Hom}
\DeclareMathOperator{\Proj}{Proj}
\DeclareMathOperator{\Spec}{Spec}

\DeclareMathOperator{\Image}{Image}

\DeclareMathOperator{\Aut}{Aut}
\DeclareMathOperator{\Ann}{Ann}
\DeclareMathOperator{\Supp}{Supp}

\DeclareMathOperator{\GL}{GL}
\newcommand{\FF}{\mathbb{F}}

\newcommand{\PP}{\mathbb{P}}
\newcommand{\ZZ}{\mathbb{Z}}
\newcommand{\QQ}{\mathbb{Q}}

\newcommand{\CC}{\mathbb{C}}

\newcommand{\NN}{\mathbb{N}}

\newcommand{\mm}{\mathfrak{m}}
\newcommand{\pp}{\mathfrak{p}}
\newcommand{\supp}{\mathrm{supp}}

\newcommand{\init}{\mathrm{in}}

\usepackage{tikz}
\usetikzlibrary{decorations.markings,intersections,positioning,calc}

  \tikzset{mylabel/.style  args={at #1 #2  with #3}{
    postaction={decorate,
    decoration={
      markings,
      mark= at position #1
      with  \node [#2] {#3};
 } } } }

\numberwithin{equation}{section}
\newtheorem{theorem}{Theorem}[section]

\newtheorem{question}[theorem]{Question}

\newtheorem{cor}[theorem]{Corollary}
\newtheorem{proposition}[theorem]{Proposition}
\newtheorem{remark}[theorem]{Remark}

\newtheorem{definition}[theorem]{Definition}

 \newtheorem{example}[theorem]{Example}

\title{Herzog ideals and $F$-singularities}
\author[]{Alessandro De Stefani}
\address{Dipartimento di Matematica, Universit\'a di Genova, Italy}
\email{alessandro.destefani@unige.it}
\author[]{Linquan Ma}
\address{Math Department, Purdue University, USA}
\email{ma326@purdue.edu}
\author[]{Matteo Varbaro}
\address{Dipartimento di Matematica, Universit\'a di Genova, Italy}
\email{matteo.varbaro@unige.it}

\begin{document}
\maketitle

\begin{abstract}
In this paper we study the connection between Herzog ideals (i.e., ideals with a squarefree Gr\"obner degeneration) and $F$-singularities. More precisely, we show that, in positive characteristic, homogeneous Herzog ideals define $F$-anti-nilpotent rings, and we inquire, in characteristic 0, on a surprising relationship between being Herzog ideals after a change of coordinates and defining rings of dense {\it open} $F$-pure type.
\end{abstract}

\section{Introduction}
Let $S=k[X_0,\dots,X_n]$ be a polynomial ring over a field $k$. Following \cite{HTVW}, an ideal $I\subseteq S$ is called a {\it Herzog ideal} if there exists a monomial order $<$ (i.e., a total order on the monomials of $S$ such that $1<v$ and $u<v\implies uw<vw$ where $u,v,w$ are arbitrary monomials of $S$ with $v\neq 1$) such that $\init_<(I)$ is squarefree. The class of Herzog ideals is largely populated:  It includes ideals defining Algebras with Straightening Law, Knutson ideals, Cartwright-Sturmfels ideals, 
and many more. 
The name comes from Herzog's conjecture, resolved in \cite{CoVa}, saying that the connection between $I$ and $\init_<(I)$ is much tighter than usual if $\init_<(I)$ is squarefree. 


The first purpose of this paper is to study the connection between Herzog ideals and $F$-singularities; this has already been investigated in \cite{KoVa}, where it was proved, for example, that in positive characteristic Herzog ideals define $F$-full and $F$-injective rings, see also \cite{G-M} for the latter result. The proofs of these results are based on ideas in \cite{Ma14}, where the stronger notion of $F$-anti-nilpotent has been introduced. An obstacle to proving $F$-anti-nilpotency of quotients by Herzog ideals is that it is not known whether this property localizes, see Remark \ref{Erratum}. In Section 2 we will prove that {\it homogeneous} (not necessarily with respect to a standard grading) Herzog ideals $I\subseteq S$ define $F$-anti-nilpotent rings.

\begin{theorem}[{Theorem~\ref{t:HF-anti}}]
Let $R$ be an $\NN$-graded algebra over an $F$-finite field $k$ of prime characteristic $p$ with homogeneous maximal ideal $\mm$. Write $R=S/I$ where $S$ is an $\NN$-graded polynomial ring over $k$. If $I$ is a Herzog ideal, then $R_{\mm}$ is $F$-anti-nilpotent.
\end{theorem}

We next investigate the relationship between Herzog ideals and $F$-purity. Note that, $F$-purity implies $F$-anti-nilpotency by the main result of \cite{Ma14} (see also Theorem~\ref{t:F-pure->anti} for a stronger statement). In general, Herzog ideals do not necessarily define $F$-pure quotients (see Remark~\ref{r:Herzog->open F-pure}). However, our experiments may suggest some surprising connection (in characteristic $0$) between Herzog ideals after a change of coordinates and ideals defining $F$-pure quotients for {\it all} primes $p\gg 0$. More precisely, we ask and study the following.

\begin{question}\label{q:genovaCY}
Let $I\subseteq \QQ[X_0,\dots,X_n]$ be a homogeneous ideal. Consider the following two conditions:
\begin{enumerate}
\item $\FF_p[X_0,\dots,X_n]/I_p$ is $F$-pure for all primes $p\gg 0$.
\item Possibly after a change of variables, $I_\CC\subseteq \CC[X_0,\dots,X_n]$ is a Herzog ideal.
\end{enumerate}
Under what assumptions are the two conditions above equivalent?
\end{question}

Of course, one cannot expect the implication $(2)\implies (1)$ without Gorenstein-type assumptions (see Remark~\ref{r:Herzog->open F-pure}). Also the implication $(1)\implies (2)$ is not true in general, see Remark \ref{r:counting}. On the other hand, our results in Section 3 address Question~\ref{q:genovaCY} in the following cases:

\begin{theorem}[{Proposition~\ref{p:curves}}]
Let $I\subseteq \QQ[X_0,\dots,X_n]$ be a homogeneous ideal such that $I_\CC$ defines a projective (connected) nonsingular curve $C\subseteq \PP^n$ containing at least one $\QQ$-rational point $P\in C$. Then $(1)$ and $(2)$ in Question~\ref{q:genovaCY} are equivalent.
\end{theorem}

\begin{theorem}[{Theorem~\ref{t:genova}}]
Let $f\in \QQ[X_0,\dots,X_n]$ be a homogeneous polynomial of degree $\leq 3$. Set $I=(f)$ and $H=\Proj \left( \CC[X_0,\dots,X_n]/I_\mathbb{C} \right)$. Then $(1)$ and $(2)$ in Question~\ref{q:genovaCY} are equivalent assuming either $H$ is klt or $H$ is a curve.
\end{theorem}

Finally, in Section 4, we prove some miscellaneous results on the annihilators of $F$-stable submodules of local cohomology modules. We will prove an extension of the main theorem of \cite{Ma14}, and we will prove the following generalization of some results in \cite{DSGNB}.

\begin{theorem}[{Proposition~\ref{p:radcomp}}]
Let $(R,\mm)$ be an $F$-finite complete local ring of prime characteristic $p$, and $i \in \ZZ$. If the annihilator $J$ of an $F$-stable subquotient of $H_{\mm}^i(R)$ is radical (e.g., if the Frobenius action on that subquotient is injective), then $J$ is a uniformly compatible ideal. In particular, if $R$ is $F$-anti-nilpotent, then the annihilator of any $F$-stable subquotient of $H_{\mm}^i(R)$ is uniformly compatible.
\end{theorem}

\subsection*{Acknowledgments}
The first and third authors were partially supported by the MIUR Excellence Department Project CUP D33C23001110001, PRIN 2022 Project 2022K48YYP, and by INdAM-GNSAGA. The second author was partially supported by NSF grants DMS-2302430 and DMS-2424441 when preparing this article. The second author would like to thank Ziquan Zhuang for discussions on Question~\ref{q:genovaCY} $(1)\implies (2)$. The second author would also like to thank Karl Schwede for first pointing out the gap in \cite[Theorem 4.6]{Ma14}, and for many insightful discussions on $F$-singularities throughout the years. We have used OpenAI's ChatGPT 5.6 to correct some inaccuracies and some grammar mistakes.

\section{Positive characteristic preliminaries} 
Throughout this paper, unless otherwise stated, all rings will be commutative, Noetherian, with multiplicative identity $1$. Let $R$ be a ring of prime characteristic $p$.
For $e \in \NN$, let $F^e:R \to R$ be the $e$-th iterate of the Frobenius endomorphism on $R$, that is, the $p^e$-th power map. For an $R$-module $M$, we denote by $F^e_*M$ the corresponding $R$-module given by restriction of scalars via $F^e$.

\begin{definition}
With notations as above, we say $R$ is 
\begin{itemize}
\item {\it $F$-finite} if $F^e_*R$ is a finitely generated $R$-module for some (equivalently, all) $e \in \NN$.
\item {\it $F$-pure} if the Frobenius map is pure, i.e., if for any $R$-module $M$ the induced map $M \to F_*R \otimes_R M$ is injective.
\item {\it $F$-split} if the Frobenius map splits, i.e., if there exists $\phi \in \Hom_R(F_*R,R)$ such that $\phi(1) = 1$.
\end{itemize}
\end{definition}
It is easy to see that $F$-split always implies F-pure. Moreover, the two notions are equivalent for $F$-finite rings and for complete local rings, but they may differ in general. If $R$ is a finitely generated algebra over a field $k$ or more generally a complete local ring with coefficient field $k$, we point out that $R$ is $F$-finite if and only if $[k:k^p]< \infty$. We refer the readers to \cite[Chapters 1 and 2]{MaPolstraFsingularitesBook} for these facts.

\subsection{Modules with a Frobenius action} We use the same notations as above.  

\begin{definition} A {\it Frobenius action} on an $R$-module $M$ is an additive map $F: M \to M$ such that $F(rx) = r^pF(x)$ for all $r \in R$ and all $x \in M$. An $R$-submodule $N \subseteq M$ is said to be {\it $F$-stable} if $F(N) \subseteq N$. 
\end{definition}

Note that giving a Frobenius action on $M$ is the same as giving an $R$-linear map $M \to F_*M$. For $e \in \NN$, let $\mathcal{F}^e_R(-)$ be the Frobenius functor of Peskine-Szpiro on the category of $R$-modules, that is, the functor defined as the base change to $F^e_*R$ followed by the identification of $F^e_*R$ with $R$. It is easy to see that giving a Frobenius action on $M$ is equivalent to giving an $R$-linear map $\mathcal{F}_R(M) \to M$: before identifying $F_*R$ with $R$, this map is given by $F_*r \otimes x \mapsto rF(x)$.

Let $R[F]$ be the Frobenius skew polynomial ring, i.e., the non-commutative ring generated over $R$ by the symbols $1, F, F^2 \ldots$ by requiring that $Fr=r^pF$ for all $r \in R$. Then having a Frobenius action is the same as being a left $R[F]$-module; moreover an $R$-submodule is $F$-stable if and only if it is a left $R[F]$-submodule.\footnote{In this article we will mostly deal with left $R[F]$-modules. Whenever not specified, an $R[F]$-module will always mean a left $R[F]$-module.}

\begin{definition} If $R$ is $\NN$-graded, then a graded $R[F]$-module $M$ is a graded $R$-module $M$ such that $F(M_d) \subseteq M_{pd}$ for all $d \in \ZZ$.
\end{definition}

Note that if $M$ is a graded $R[F]$-module then the associated $R$-linear map $\mathcal{F}_R(M) \to M$ is degree preserving.

\begin{definition} Let $M$ be an $R[F]$-module. We say that $M$ is
\begin{itemize}
\item {\it full} if the map $\mathcal{F}_R^e(M) \to M$ is surjective for some (equivalently, all) $e \in \NN$.
\item {\it anti-nilpotent} if, for any $R[F]$-submodule $N \subseteq M$, the induced Frobenius action $F:M/N \to M/N$ is injective.
\end{itemize}
\end{definition}

We remark that $M$ is anti-nilpotent if and only if every $R[F]$-submodule of $M$ is full, see \cite[Lemma 2.1]{MaQuyDeformation}. In particular, if $M$ is anti-nilpotent then it is full.

\subsection{Local cohomology modules and $F$-singularities}
Given an ideal $I=(f_1,\ldots,f_t) \subseteq R$ and an $R$-module $M$, we recall that the $i$-th local cohomology module $H^i_I(M)$ can be defined as the $i$-th cohomology of the \v{C}ech complex
\[
\xymatrix{
0 \ar[r] & M \ar[r] & \bigoplus_{i=1}^t M_{f_i} \ar[r] & \bigoplus_{1 \leq i < j \leq t} M_{f_if_j} \ar[r] & \ldots\ldots  \ar[r] & M_{f_1\cdots f_t} \ar[r] & 0.
}
\]
Moreover, the Frobenius map on $R$ induces maps $F^e:H^i_I(R) \to H^i_{I^{[p^e]}}(R) \cong H^i_I(R)$ for all $i \in \ZZ$ and all $e \in \NN$, where $I^{[p^e]} = F^e(I)R = (f_1^{p^e},\ldots,f_t^{p^e})$ denotes the $e$-th Frobenius power of $I$. Equivalently, they can be seen as $R$-linear maps $H^i_I(R) \to H^i_I(F^e_*R) \cong F^e_*(H^i_I(R))$.

\begin{definition} Let $(R,\mm)$ be a local ring of prime characteristic $p$. We say that $R$ is 
\begin{itemize}
\item {\it $F$-injective} if the induced Frobenius maps $F:H^i_\mm(R) \to H^i_\mm(R)$ are injective for all $i \in \ZZ$.
\item {\it $F$-full} if $H^i_\mm(R)$ is a full $R[F]$-module for all $i \in \ZZ$.
\item {\it $F$-anti-nilpotent} if $H^i_\mm(R)$ is an anti-nilpotent $R[F]$-module for all $i \in \ZZ$.
\end{itemize}
\end{definition}

By the discussions above, $F$-anti-nilpotent local rings are both $F$-full and $F$-injective. It is easy to see that $F$-full rings need not be $F$-injective; in fact, any Cohen-Macaulay local ring is $F$-full but not necessarily $F$-injective (e.g., $F$-injective rings are necessarily reduced, see \cite[Lemma 3.11]{QuyShimomoto}). Conversely, there are examples of F-injective rings which are not $F$-full \cite[Example 3.5]{MaSchwedeShimomoto}. In relation with the $F$-singularities defined above, we recall that $F$-pure local rings are $F$-anti-nilpotent by \cite[Theorem 1.1]{Ma14}, but the converse does not hold in general, see \cite[Sections 5 and 6]{QuyShimomoto}.

While the notions of being $F$-injective and $F$-full localize, see \cite[Proposition 3.3]{DaMuF-inj} and \cite[Proposition 2.7]{MaQuyDeformation}, it is not known if the same is true for the $F$-anti-nilpotent property:

\begin{remark}\label{Erratum}
It is not known whether $F$-anti-nilpotency localizes. This localization property (and, in fact, a slightly stronger statement) was claimed in \cite[Theorem 4.6 and Theorem 5.10]{Ma14}. However, the proof of \cite[Theorem 4.6 $(2)\Rightarrow(1)$]{Ma14} contains a gap: the specific error in the argument was that, after localization, the completeness assumption was lost so one cannot use 
\cite[Theorem 4.4]{Ma14}. In the $F$-finite local case, the Matlis dual of each Cartier submodule of $H^{-i}(\omega^\bullet_R)$ is an $F$-stable quotient of $H_{\mm}^i(R)$ and thus gives an $F$-stable submodule of $H_{\mm}^i(R)$. Hence, the implication \cite[Theorem 4.6 $(1)\Rightarrow(2)$]{Ma14} still holds, and the rest of the results in \cite[Section 5]{Ma14} (except \cite[Theorem 5.10]{Ma14}) are unaffected by this mistake.
\end{remark}

The gap in \cite[Theorem 4.6]{Ma14} was first pointed out to the second author by Karl Schwede, who proposed to use the notion of Cartier-anti-nilpotency instead of $F$-anti-nilpotency to reconcile the localization property.

%
%
%
%

\begin{definition}
An F-finite ring $R$ of prime characteristic $p$ is called Cartier-anti-nilpotent if $H^{-i}(\omega_R^\bullet)$ is an anti-nilpotent Cartier module (in the sense of \cite[Definition 1.19 on page 491]{SchwedeSmithBook}) for each $i$, where $\omega_R^\bullet$ is the dualizing complex of $R$.
\end{definition}

We record some basic facts about Cartier-anti-nilpotency, see \cite[Lemma 1.20 and Lemma 1.21 on page 491-492]{SchwedeSmithBook}.
\begin{itemize}
  \item For an $F$-finite local ring $(R,\mm)$, if $R$ is $F$-anti-nilpotent, then $R$ is Cartier-anti-nilpotent.
  \item An $F$-finite ring $R$ is Cartier-anti-nilpotent if and only if all localizations of $R$ are Cartier-anti-nilpotent.
\end{itemize}

The following question is open; an affirmative answer would imply that $F$-anti-nilpotency localizes. See \cite[Question 1.27 on page 495]{SchwedeSmithBook} for related questions.

\begin{question}
Let $(R,\mm)$ be an $F$-finite local ring of prime characteristic $p$. If $R$ is Cartier-anti-nilpotent, then is it $F$-anti-nilpotent?
\end{question}

In Proposition \ref{GradedCartier} we will prove that, in the graded situation, the above question has an affirmative answer. We need the
following graded version of \cite[Theorem 4.7]{LyubeznikFmodules} which is well-known to experts. We include a proof here for completeness using the theory of (graded) $\mathcal{F}$-modules, see \cite{LyubeznikFmodules} and \cite{LSW} for their definitions and basic properties.

\begin{proposition} \label{prop graded Lyubeznik filtration}
Let $R$ be an $\NN$-graded algebra over a field $k$ of prime characteristic $p$. For any graded Artinian $R[F]$-module $W$, there exists a filtration of graded $R[F]$-modules 
\[
0=L_0\subseteq N_0 \subseteq L_1 \subseteq N_1 \subseteq \cdots \subseteq L_n \subseteq N_n =W
\]
such that each $L_i/N_{i-1}$ is a simple graded $R[F]$-module with nontrivial Frobenius action, and each $N_i/L_i$ is a graded $R[F]$-module with nilpotent Frobenius action.
\end{proposition}
\begin{proof}
First of all, we set $L_n:=\langle F^e(W)\rangle$ for $e\gg0$. Then $L_n$ is graded and full, and $W/L_n$ is nilpotent by construction. We next take any $R[F]$-submodule $N_{n-1}\subseteq L_n$ so that $L_n/N_{n-1}$ is a simple $R[F]$-module, the existence of $N_{n-1}$ is guaranteed by \cite[Theorem 4.7]{LyubeznikFmodules} and the Frobenius action on $L_n/N_{n-1}$ is necessarily nontrivial by the fullness of $L_n$ (in particular $L_n/N_{n-1}$ is full). Note that, a priori, we do not know whether $N_{n-1}$ is graded, but we will show that we can always choose a graded $N_{n-1}$. To see this, write $R=S/I$ where $S$ is an $\mathbb{N}$-graded polynomial ring over $k$ and let $(-)^\vee$ be the graded Matlis duality, we know that 
$$\mathcal{H}(L_n) : =\varinjlim (L_n^\vee \hookrightarrow \mathcal{F}_S(L_n^\vee) \hookrightarrow \mathcal{F}^2_S(L_n^\vee) \hookrightarrow \cdots )$$
is a graded $\mathcal{F}_S$-module (where the maps in the direct limit are injective by the fullness of $L_n$) with root $L_n^\vee$. By \cite[Proposition 2.2]{LSW},\footnote{The proof of \cite[Proposition 2.2]{LSW} requires standard graded, we leave it to the readers to check that with small modification the conclusion holds for any $\mathbb{N}$-grading.} any simple $\mathcal{F}_S$-submodule of $\mathcal{H}(L_n)$ is graded so we fix one and call it $\mathcal{M}$. It is easy to see that $M:=\mathcal{M}\cap L_n^\vee$ is a root of $\mathcal{M}$ and it is a graded submodule of $L_n^\vee$ by construction. Now set $N'_{n-1}:=(L_n^\vee/M)^\vee$, we have that $N'_{n-1}$ is graded and $L_n/N'_{n-1}=M^\vee$. Then we set $N_{n-1}\supseteq N_{n-1}'$ be the graded $R[F]$-module so that $L_n/N_{n-1}=M^\vee/(M^\vee)_{\text{nil}}$, where $(M^\vee)_{\text{nil}}$ is the nilpotent part of $M^\vee$. It is then readily checked that $L_n/N_{n-1}$ is a simple $R[F]$-module. We can now continue this procedure: set $L_{n-1}=\langle F^e(N_{n-1})\rangle$ for $e\gg0$ and choose a graded $R[F]$-submodule $N_{n-2}$ of $L_{n-1}$ so that $L_{n-1}/N_{n-2}$ is a simple $R[F]$-module, etc. This process must terminate after finitely many steps by \cite[Theorem 4.7]{LyubeznikFmodules}. Thus we obtain the desired filtration of graded $R[F]$-modules. 
\end{proof}

\begin{proposition}
\label{GradedCartier}
Let $R$ be an $\NN$-graded algebra over an $F$-finite field of prime characteristic $p$. Then the ideal defining the non-Cartier-anti-nilpotent locus of $R$ is homogeneous. 

Moreover, if $\mm$ is the unique maximal homogeneous ideal of $R$, the following conditions are equivalent:
\begin{enumerate}
  \item $R$ is Cartier-anti-nilpotent;
  \item $R_\mm$ is Cartier-anti-nilpotent;
  \item $R_\mm$ is F-anti-nilpotent.
\end{enumerate}
\end{proposition}
\begin{proof}

By Proposition \ref{prop graded Lyubeznik filtration} applied to the graded $R[F]$-modules $H_\mm^i(R)$, we obtain filtrations of graded $R[F]$-modules:
\[
0=L^i_0\subseteq N^i_0 \subseteq L^i_1 \subseteq N^i_1 \subseteq \cdots \subseteq L^i_{n_i} \subseteq N^i_{n_i} =H_\mm^i(R)
\]
with $N^i_j/L^i_{j}$ nilpotent and $L^i_j/N^i_{j-1}$ simple graded $R[F]$-modules. Applying graded Matlis duality yields filtrations of graded Cartier submodules of $H^{-i}(\omega_R^\bullet)$:
$$0=C^i_0\subseteq D^i_0 \subseteq C^i_1 \subseteq D^i_1 \subseteq \cdots \subseteq C^i_{n_i} \subseteq D^i_{n_i} =H^{-i}(\omega_R^\bullet).$$
where $D^i_j/C^i_j$ are Cartier-nilpotent. Since $(\omega_R^\bullet)_P$, appropriately normalized, is a normalized dualizing complex of $R_P$ for every $P\in\Spec(R)$. We obtain that the non-Cartier-anti-nilpotent locus is precisely 
$$\bigcup_{i,j}\Supp_R(D^i_j/C^i_j).$$
In particular, its defining ideal is homogeneous since all $C^i_j$, $D^i_j$ are finitely generated graded modules. For the last assertion, $(1)\Leftrightarrow(2)$ follows immediately from the already established statement, and $(3)\Rightarrow(2)$ always holds. To see $(2)\Rightarrow(3)$, simply notice that if $D^i_j/C^i_j=0$ for all $i,j$ then $N^i_j/L^i_j=0$ by graded Matlis duality, thus $R_\mm$ is $F$-anti-nilpotent.
\end{proof}

\begin{remark}
In Proposition~\ref{GradedCartier}, we do not know whether the equivalent conditions $(1)-(3)$ implies that $R_P$ is $F$-anti-nilpotent for all $P\in\Spec(R)$. The same issue in establishing the localization property of $F$-anti-nilpotency still occurs: after localization we lose the graded assumption so it is not clear that every $F$-stable submodule of the local cohomology of $R_P$ arises from the Matlis dual of a Cartier submodule of the cohomology of $\omega^\bullet_{R_P}$.   
\end{remark}

\subsection{Herzog ideals and $F$-anti-nilpotency} We can now prove the first result of this paper, namely that homogeneous Herzog ideals define $F$-anti-nilpotent quotients after localizing at the homogeneous maximal ideal.

\begin{theorem}\label{t:HF-anti}
Let $R$ be an $\NN$-graded algebra over an $F$-finite field $k$ of prime characteristic $p$ with homogeneous maximal ideal $\mm$. Write $R=S/I$ where $S$ is an $\NN$-graded polynomial ring over $k$. If $I$ is a Herzog ideal, then $R_{\mm}$ is $F$-anti-nilpotent.
\end{theorem}
\begin{proof}
Let us choose a monomial order on $S$ so that $\init(I)$ is squarefree. By \cite[Theorem 15.17]{EisenbudBook}, we can find a homogeneous ideal $\widetilde{I}\subseteq S[t]$ so that 
$$S[t]/(\widetilde{I}, t)\cong S/\init(I) \text{ and } (S[t]/\widetilde{I})\otimes_{k[t]}k(t)\cong R\otimes_kk(t).$$

Since $\init(I)$ is squarefree, $S/\init(I)$ is $F$-pure and thus $(S/\init(I))_\mm$ is $F$-anti-nilpotent by \cite[Theorem 3.8]{Ma14}. It follows from \cite[Theorem 4.2]{MaQuyDeformation} that $(S[t]/\widetilde{I})_{\mm+(t)}$ is $F$-anti-nilpotent and thus Cartier-anti-nilpotent. Hence $S[t]/\widetilde{I}$ is Cartier-anti-nilpotent by Proposition~\ref{GradedCartier}. But then $(S[t]/\widetilde{I})\otimes_{k[t]}k(t)\cong R\otimes_kk(t)$ is Cartier-anti-nilpotent as this property is stable under localization. It follows that $(R\otimes_kk(t))_\mm$ is $F$-anti-nilpotent by Proposition~\ref{GradedCartier}. Since $H_\mm^i(R_\mm)\otimes_kk(t)\cong H_\mm^i((R\otimes_kk(t))_\mm)$, it is easy to see that the Frobenius action is anti-nilpotent on $H_\mm^i(R_\mm)$ (as this is true after base change along $k\to k(t)$), thus $R_\mm$ is $F$-anti-nilpotent.
\end{proof}


\begin{cor}
Let $R$ be an $\NN$-graded Algebra with Straightening Law (ASL) over an $F$-finite field $k$ of prime characteristic $p$. Then $R_{\mm}$ is $F$-anti-nilpotent, where $\mm=R_+$.
\end{cor}
\begin{proof}
We have that $R=S/I$ where $S$ is an $\NN$-graded polynomial ring over $k$, and $\init(I)$ is squarefree for any DegRevLex monomial order extending the partial order on the variables given by the poset governing the ASL structure of $R$. So the result follows by Theorem~\ref{t:HF-anti}.
\end{proof}

\section{Herzog ideals versus $F$-purity for all $p\gg 0$}

Throughout this section any polynomial ring $S=A[X_0,\dots ,X_n]$ over a commutative ring $A$ will be standard graded, i.e., $A=S_0$ and $\deg(X_i)=1$ for all $i=0,\dots ,n$. In this section we study Question~\ref{q:genovaCY}. We first investigate the case of nonsingular curves. 


\begin{proposition}\label{p:curves}
Let $I\subseteq \QQ[X_0,\dots,X_n]$ be a homogeneous ideal such that $I_\CC$ defines a projective (connected) nonsingular curve $C\subseteq \PP^n_{\CC}$ containing at least one $\QQ$-rational point $P\in C$. 
Then the following conditions are equivalent:
\begin{enumerate}
\item $\FF_p[X_0,\dots,X_n]/I_p$ is $F$-pure for all primes $p\gg 0$.
\item Possibly after a change of variables, $I_\CC\subseteq \CC[X_0,\dots,X_n]$ is a Herzog ideal.
\end{enumerate}
\end{proposition}
\begin{proof}
We first note that under either $(1)$ or $(2)$, the embedding $C\subseteq \PP^n_{\CC}$ must be projectively normal. In fact, let $R=\CC[X_0,\dots,X_n]/I_\CC$, we know that $H_{\mm}^1(R)_{\leq 0}=0$, for example see \cite[Proof of Theorem 5.9]{DaoMaVarbaro}, while either $(1)$ or $(2)$ implies that $R$ has dense $F$-injective type and thus $H_{\mm}^1(R)_{>0}=0$. Therefore we know that $R$ has depth at least two and is nonsingular on the punctured spectrum, which implies $R$ is normal.

$(2)\implies (1)$. Since $C\subseteq \PP^n$ is nonsingular and $I_\CC$ is a Herzog ideal, we know by \cite{HTVW} that $C$ has genus $0$. Since we know the embedding is projectively normal, it must be a rational normal curve. So (1) follows because $\FF_p[X_0,\dots,X_n]/I_p$ is, for $p\gg 0$, a direct summand of $\overline{\FF_p}[X^iY^{n-i}:0\leq i\leq j\leq n]$, which is in turn a direct summand of $\overline{\FF_p}[X,Y]$ that is obviously $F$-pure.

$(1)\implies (2)$. Our assumption guarantees that $C$ has genus at most 1 (see \cite[Remark 1.2.10]{BrionKumarFsplittingBook}). Assume that $C$ has genus $1$. Since it has a $\QQ$-rational point, $C$ is then an elliptic curve over $\QQ$. By \cite{ElkiesSupersingular}, there are then infinitely many supersingular primes $p$, and the homogeneous coordinate ring $\FF_p[X_0,\dots,X_n]/I_p$ is thus not $F$-pure for infinitely many primes $p$ by \cite[Remark 1.3.9]{BrionKumarFsplittingBook}. This contradicts $(1)$. So $C$ has genus $0$, and since we know the embedding is projectively normal, $C\subseteq \PP^n$ is a rational normal curve of degree $d \leq n$. So after a change of variables $I_\CC$ is the ideal of $2$-minors of the matrix (together with the remaining variables)
\[U=\begin{pmatrix}
X_0 &X_1 & \dots & X_{d-1} \\
X_1 &X_2 & \dots & X_d
\end{pmatrix},\]
and $\init(I_2(U))=(X_iX_j: 1\leq i+1<j\leq d)$ for LEX $X_0>X_1>\dots >X_n$.   
\end{proof}

\begin{remark}
\label{r:Herzog->open F-pure}
\begin{itemize}
  \item We do not know whether the implication ``$(2)\implies (1)$" in Question~\ref{q:genovaCY} holds when $I_{\CC}$ defines a singular projective curve. 
  \item On the other hand, ``$(2)\implies (1)$" does not hold already if $I_\CC$ defines a projective surface. For example, the ideal 
$$I=(XY,XZ,Y(ZU-W^2))\subseteq \QQ[X,Y,Z,U,W]$$ 
is such that $R_p=\FF_p[X,Y,Z,U,W]/I_p$ is not $F$-pure for any prime $p$, however we have $\init(I_\CC)=(XY,XZ,YZU)$ for LEX with $X>Y>Z>U>W$. To see that $R_p$ is not $F$-pure one can argue as follows: if $u=\overline{U}\in R_p$, the localization $(R_p)_u$ is isomorphic to $A_p[U,U^{-1}]$ where $A_p=\FF_p[X,Y,Z,W]/(XY,XZ,Y(Z-W^2))$ (e.g., using \cite[Proposition 1.5.18]{BH1993}). If $R_p$ is $F$-pure so is $A_p$; however, $A_p$ is not $F$-pure for any prime integer $p$ by \cite[Example 3.2]{Anurag99}.
  \item It is well-known that ``$(2)\implies (1)$" in Question~\ref{q:genovaCY} holds when $\CC[X_0,\dots,X_n]/I_\CC$ is Gorenstein. Indeed, if $I_\CC\subseteq \CC[X_0,\dots,X_n]$ is a Herzog ideal after a change of coordinates, for all $p\gg 0$ there exists a field $k$ of characteristic $p$ and a change of variables so that $I_k\subseteq k[X_0,\dots,X_n]$ becomes a Herzog ideal. In particular, $k[X_0,\dots,X_n]/I_k$ is $F$-injective. However, for $p\gg 0$ we have that $k[X_0,\dots,X_n]/I_k$ is also Gorenstein, and hence $F$-pure. Therefore $\FF_p[X_0,\dots,X_n]/I_p$ is $F$-pure as well.
\end{itemize}
\end{remark}


We next study ``$(1)\implies (2)$'' in Question~\ref{q:genovaCY}. Note that, if $f\in \QQ[X_0,\dots,X_n]$ is a homogeneous polynomial of degree $d\leq n$ defining a nonsingular hypersurface $X\subseteq \PP^n_\CC$, then $\FF_p[X_0,\dots,X_n]/(\overline{f})$ is $F$-pure for all $p\gg 0$ by \cite[Theorem 2.5]{Fed83}. 
As a first example, consider $f=X^3+Y^3+Z^3+W^3\in \QQ[X,Y,Z,W]$ and $I=(f)$. Then $I_p\subseteq \FF_p[X,Y,Z,W]/(X^3+Y^3+Z^3+W^3)$ defines an $F$-pure quotient for all primes $p>3$. In this case, if $g\in\Aut(\CC[X,Y,Z,W])$ is defined as $g(X)=X+Y+Z$, $g(Y)=-(X+Y)$, $g(Z)=-(X+Z)$, and $g(W)=X+W$, then it turns out that 
\[g(f)=6XYZ+3Y^2Z+3YZ^2+3X^2W+3XW^2+W^3,\]
so $\init(g(f))=XYZ$ with respect to DegRevLex with $X>Y>Z>W$. We now show that this is indeed a property common to all klt cubic hypersurfaces.

\begin{theorem}
\label{t:genova}
Let $f\in \QQ[X_0,\dots,X_n]$ be a homogeneous polynomial of degree $d \leq 3$. Set $I=(f)$ and $H=\Proj \CC[X_0,\dots,X_n]/I_\mathbb{C}$. Assume either $H$ has klt singularities\footnote{This condition implies that $\CC[X_0,\dots,X_n]/I_\mathbb{C}$ is klt when $n\geq 3$ by \cite{KSSW} or \cite{WatanabeRationalSingularities} (for hypersurfaces, klt is equivalent to rational singularities) and thus is $F$-pure for all primes $p\gg 0$ by \cite[Theorem 5.2]{HaraRationalSingularities}.} or $H$ is a curve, then the following conditions are equivalent:
\begin{enumerate}
\item $\FF_p[X_0,\dots,X_n]/I_p$ is $F$-pure for all primes $p\gg 0$.
\item Possibly after a change of variables, $I_\CC\subseteq \CC[X_0,\dots,X_n]$ is a Herzog ideal.
\end{enumerate}
\end{theorem}
\begin{proof}
First of all, since $\CC[X_0,\dots,X_n]/I_\mathbb{C}$ is a hypersurface and in particular Gorenstein, $(2)\implies (1)$ is a special case of the last item of Remark~\ref{r:Herzog->open F-pure}. We thus focus on $(1)\implies (2)$. If $d=1$ the conclusion is trivial, and if $d=2$ it immediately follows by the classification of quadrics. In what follows we assume that $d=3$.

In the case that $H$ is a curve, then \cite[Theorem 2]{ElkiesSupersingular} (and the assumption (1)) guarantees that $H$ is necessarily a singular projective curve in $\PP^2_\CC$. We can assume that $P=[1:0:0]$ is a singular point of $H$, that is, we can do a change of coordinates $\alpha\in\Aut(\CC[X,Y,Z])$ such that $X^3\notin \supp(\alpha(f))$ and 
$$\frac{\partial (\alpha(f))}{\partial X}(P)=\frac{\partial (\alpha(f))}{\partial Y}(P)=\frac{\partial (\alpha(f))}{\partial Z}(P)=0.$$ This means that $\alpha(f)=Xq+g$, where $q$ is a quadric of $\CC[Y,Z]$ and $g$ is a cubic of $\CC[Y,Z]$. As before, there exists a change of coordinates $\beta\in\Aut(\CC[Y,Z])$ such that $\beta(q)=0$, $\beta(q)=Y^2$ or $\beta(q)=YZ$: in the third case $\init(\beta\circ \alpha(f))=XYZ$ for LEX with respect to $X>Y>Z$; in the second and first case, it is easy to see that $f^{p-1}$ cannot contain $(XYZ)^{p-1}$ in its support thus by Fedder's criterion $\FF_p[X,Y,Z]/(\overline{f})$ would not be $F$-pure and thus contradicting $(1)$. 

Now we assume that $H$ has klt singularities, and we may assume that $n\geq 3$ (if $n=2$ then $H$ is a curve and we have already established this case).  

First we assume $n=3$: in this case $H\subseteq \PP^3$ is a cubic surface. If $H$ is nonsingular, then there exist linear forms $l_i,m_i\in S=\CC[X_0,X_1,X_2,X_3]$ where $i=1,2,3$ such that 
\[f=l_1l_2l_3-m_1m_2m_3.\] 
This goes back to the 19th century and is known as Cayley-Salmon equation, see \cite{CayleySalmon} for a modern treatment. The linear forms $l_i,m_i$ for $i=1,2,3$ are two by two linearly independent (see \cite[Section 2.3]{CayleySalmon}) and $(l_i,m_i:i=1,2,3)=(X_0,X_1,X_2,X_3)$ (otherwise the point corresponding to $(l_i,m_i:i=1,2,3)$ would be a singular point of $H$).  
If $\dim_\CC\langle l_1,l_2,l_3\rangle = \dim_\CC\langle m_1,m_2,m_3\rangle =2$, then $f$ must be the Fermat cubic up to change of variables, and this case has already been discussed just before the theorem. Otherwise, up to relabeling
$l_1,l_2,l_3$ and $m_1$ are linearly independent, so there exists a change of variables $\phi \in \Aut(S)$ such that
\[\phi(f)=X_0X_1X_2-X_3\phi(m_2)\phi(m_3).\]
In particular $\init(\phi(f))=X_0X_1X_2$ for DegRevLex with $X_0>X_1>X_2 >X_3$.

If $H$ is a singular cubic surface, then the defining equations of such $H$ are classified by \cite[Table 2 and Theorem 2]{SakamakiCubicSurfaces} (all but the last equation define klt hypersurface $H$). From those explicit equations, it is clear that one can choose a monomial order so that the initial term is squarefree except the following three cases:
\begin{enumerate}
  \item $f=X_3X_0^2 + X_1^3 +X_2^3$
  \item $f=X_3X_0^2 + X_0X_2^2 + X_1^2X_2$
  \item $f=X_3X_0^2 + X_0X_2^2 +X_1^3$
\end{enumerate}
We now tackle these three cases by hand. In case $(1)$, we consider the change of variables $\phi\in\Aut(\CC[X_0,X_1,X_2,X_3])$ such that $\phi(X_0)=X_0+X_1$, $\phi(X_1)=-(X_1+X_3)$, $\phi(X_2)=X_1+X_2$, $\phi(X_3)=3X_3$. A straightforward computation shows that 
$$\phi(f)=6X_0X_1X_3 + 3X_0^2X_3 -3X_1X_3^2 -X_3^3 + 3X_1^2X_2 +3X_1X_2^2 +X_2^3$$
and it is clear that $\init(\phi(f))=X_0X_1X_3$ for DegRevLex with $X_1>X_0>X_3>X_2$. One can find a similar change of variables in case $(2)$ and case $(3)$, but for these two cases, one can alternatively consider the following two equations:
\begin{enumerate}
  \item[(2')] $f=X_3X_0^2 + X_0X_2^2 + X_1^2X_2 + X_0X_1X_2$
  \item[(3')] $f=X_3X_0^2 + X_0X_2^2 +X_1^3 + X_0X_1X_2$
\end{enumerate}
The equation in $(2')$ has a unique $D_5$-singularity while the equation in $(3')$ has a unique $E_6$-singularity. Therefore, the hypersurfaces defined by the equations in $(2')$ and $(3')$ must be isomorphic to the hypersurfaces defined by the equations in $(2)$ and $(3)$ respectively by \cite[Theorem 2]{SakamakiCubicSurfaces}. Since for any normal cubic surface $H$, we have $O(1)$ is isomorphic to $-\omega_H$ and thus the embedding is the complete anti-canonical embedding. Hence an abstract isomorphism preserves $O(1)$ and extends to an element of $PGL_4(\mathbb{C})$. It is left to observe that $\init(f)=X_0X_1X_2$ for DegRevLex $X_0>X_1>X_2>X_3$ in case $(2')$ and for the monomial order by first declaring DegRevLex with respect to $X_3$ and then using LEX with $X_0>X_1>X_2$.

Finally, if $n>3$, by Bertini's theorem (see \cite[Lemma 5.17]{KollarMori}), there exists a hyperplane section $H'\subseteq \PP^{n-1}$ of $H\subseteq \PP^n$ so that $H'$ still has klt singularities. After a change of variables $\alpha\in \Aut( \CC[X_0,\dots,X_n])$, we can assume that the hyperplane is $\{X_n=0\}$; in other words,
\[\alpha(f)=f'+X_nf''\]
where $f''\in S$ and $f'\in S'=\CC[X_0,\ldots ,X_{n-1}]$ is a homogeneous polynomial of degree 3 defining a klt hypersurface $H'\subseteq \PP^{n-1}$. By induction on $n$, there exists a change of variables $\phi'\in\Aut(S')$ such that $\init(\phi'(f'))=X_0X_1X_2$ for some monomial order in $X_0, X_1, \dots, X_{n-1}$. Extending $\phi'$ to $\beta\in\Aut(S)$ by putting $\beta(X_n)=X_n$, and defining $\phi=\beta\circ \alpha\in\Aut(S)$, we have $\init(\phi(f))=X_0X_1X_2$ where we extend the monomial order from $S'$ to $S$ via DegRevLex with respect to $X_n$.
\end{proof}



We end this section by pointing out that ``$(1)\implies (2)$'' in Question~\ref{q:genovaCY} does not hold for hypersurfaces of degree $n\geq 5$. We do not know if the implication holds true for hypersurfaces of degree 4.

\begin{remark}\label{r:counting}
For all $n\geq 5$, there exists $f\in \QQ[X_0, X_1,...,X_n]$ of degree $n$ such that $\FF_p[X_0,\dots,X_n]/(f)_p$ is $F$-pure for all primes $p\gg 0$ but even up to change of variables $(f) \subseteq S=\CC[X_0, X_1,...,X_n]$ is not a Herzog ideal.

\smallskip

\noindent Indeed, the set of homogeneous polynomials $g\in S$ defining a smooth hypersurface of degree $n$ in $\PP^n$ is a nonempty open subset $U\subseteq \PP^N$ where $N = {2n \choose n}-1$, so the GIT quotient $U/\!/\GL_{n+1}(\CC)$ is a quasi projective variety over $\CC$ of dimension at least $N-(n+1)^2$ (here the action is given by changing variables in $S$).

Now, let us fix some term order with $X_1>X_2>\ldots >X_n>X_0$. Given a homogeneous polynomial $g\in S$ of degree $n$ such that $\init(g)$ is squarefree, whatever the term order is, since $X_1X_2\cdots X_n$ is the biggest squarefree monomial of degree $n$ of $S$ and $uX_i > uX_j$ for any monomial $u\in S$ and $1\leq i<j$, the following monomials cannot be in $\supp(g)\setminus \init(g)$:

\[ X_1^{a_1}X_2^{a_2}...X_n^{a_n} \ \mbox{ with } \ a_1+a_2+\ldots +a_n = n \ \mbox{ and } \  a_1+\ldots +a_i \geq i \ \forall \ i\in [n].\]

As it turns out, the number of monomials as above is the $n$th Catalan number $\frac{1}{n+1}{2n \choose n}$. So, if $n\geq 5$, the dimension of the quasi projective variety $Y\subseteq \PP^N$ given by $g\in S$ such that $g$ is homogeneous of degree $n$ and $(g)$ is a Herzog ideal is 
\[\dim(Y)\leq N-\frac{1}{n+1}{2n \choose n}+1 < N-(n+1)^2 \leq \dim(U//\GL_{n+1}(\CC)).\] 
Note that $U//\GL_{n+1}(\CC)$ is defined over $\QQ$, and there is a dominant rational map $\PP^N\dasharrow U \dasharrow U//\GL_{n+1}(\CC)$. Therefore, since the rational points of $\PP^N$ are Zariski-dense in $\PP^N$, the rational points of $U//\GL_{n+1}(\CC)$ are Zariski-dense in $U//\GL_{n+1}(\CC)$ as well, so we get the thesis. 
\end{remark}



\section{Annihilators of $F$-stable submodules and uniformly compatible ideals}

In this section we prove some results regarding annihilators of $F$-stable submodules and subquotients of local cohomology modules. Our first result is a generalization of the main result of \cite{Ma14} to rings that are not necessarily $F$-pure. We recall that the trace ideal of an $R$-module $M$ is the ideal $\sum_{\phi\in\Hom_R(M,R)}\phi(M)\subseteq R$.

\begin{theorem}\label{t:F-pure->anti}
Let $(R,\mm)$ be an $F$-finite local ring of prime characteristic $p$. Then there exists $e>0$ so that the trace ideal $J_e$ of $F_*^eR$ annihilates every $F$-stable subquotient of $H_{\mm}^i(R)$ that is nilpotent. 

In particular, if $R$ is $F$-pure, then the Frobenius action on $H_{\mm}^i(R)$ is anti-nilpotent, i.e., $R$ is $F$-anti-nilpotent.
\end{theorem}
\begin{proof}
First of all, we take a Lyubeznik filtration of $H_\mm^i(R)$ (see \cite[Theorem 4.7]{LyubeznikFmodules}):
$$0=L_0\subseteq N_0 \subseteq L_1 \subseteq N_1 \subseteq \cdots \subseteq L_n \subseteq N_n =H_{\mm}^i(R)$$
of $F$-stable submodules of $H_{\mm}^i(R)$ so that each $L_i/N_{i-1}$ is a simple $R[F]$-module with nontrivial Frobenius action and each $N_i/L_i$ is an $R[F]$-module with nilpotent Frobenius action. There exists $e_0$ so that $F^{e_0}(N_i/L_i)=0$ for all $i$. It follows by an easy filtration argument that for any $F$-stable subquotient $N/L$ of $H_{\mm}^i(R)$ that is nilpotent, $F^{(n+1)e_0}(N/L)=0$. Set $e=(n+1)e_0$. For any $\phi\in \Hom_R(F^e_*R, R)$, we need to show that $\phi(F^e_*R)$ annihilates $N/L$. We may replace $L$ by $L':=\langle F^{e'}(L)\rangle$ for $e'\gg0$: the Frobenius action on $L/L'$ is nilpotent and thus so is $N/L'$, and clearly if $\phi(F^e_*R)$ annihilates $N/L'$ then it also annihilates $N/L$. After this replacement, we have that $L$ is full. We now consider the following commutative diagram: 
\vspace{0.4em}
\[
\xymatrix{
H_{\mm}^i(R) \ar[r] & F^e_*R \otimes H_{\mm}^i(R) \ar@/^2pc/[rr]^{\phi\otimes\text{id}} \ar[r] & H_{\mm}^i(F^e_*R) \ar[r]^{H_{\mm}^i(\phi)}  & H_{\mm}^i(R) \\
N \ar[r] \ar@{^{(}->}[u]  & F^e_*R \otimes N \ar[r] \ar[u]   & F^e_*N \ar@{^{(}->}[u]\\
L \ar[r] \ar@{^{(}->}[u] & F^e_*R \otimes L \ar@{->>}[r] \ar[u] & F^e_*L  \ar@{^{(}->}[u]
}\]
where the top part is commutative by \cite[Lemma 3.3]{Ma14} (see also \cite[Lemma 1.23 on page 493]{SchwedeSmithBook}), and the bottom right surjectivity follows from the fact that $L$ is full. By our choice of $e$, we know that 
$$\Image(F^e_*R \otimes N\to F^e_*N) \subseteq F^e_*L = \Image(F^e_*R \otimes L\to F^e_*L).$$ 
Chasing the diagram, we find that 
$$\Image(F^e_*R \otimes N \to H_{\mm}^i(R)) \subseteq \Image(F^e_*R\otimes L \to H_{\mm}^i(R)).$$
Thus by the commutative diagram (use the arrow $\phi\otimes \text{id}$), we have that 
$$\phi(F^e_*R)N\subseteq L$$
as wanted.
\end{proof}

We next recall that for an $F$-finite ring $R$, an ideal $I\subseteq R$ is said {\it uniformly compatible} if $\phi(F^e_*I)\subseteq I$ for all $e\in\mathbb{N}$ and all $\phi\in \Hom_R(F^e_*R,R)$.

\begin{proposition}
\label{prop: F-full submodules}
Let $(R,\mm)$ be an $F$-finite local ring of prime characteristic $p$ and $N\subseteq H_{\mm}^i(R)$ be an $F$-stable submodule that is full. Then for any $F$-stable submodule $L$ that contains $N$, $J:=\Ann_R(L/N)$ is a uniformly compatible ideal. 

In particular, the annihilator of any $F$-stable submodule $L\subseteq H_{\mm}^i(R)$ is a uniformly compatible ideal.
\end{proposition}
\begin{proof}
Let $\phi\in \Hom_R(F^e_*R, R)$. We consider the following commutative diagram 
\vspace{0.4em}
\[
\xymatrix{
H_{\mm}^i(R) \ar[r] & F^e_*R \otimes H_{\mm}^i(R) \ar@/^2pc/[rr]^{\phi\otimes\text{id}} \ar[r] & H_{\mm}^i(F^e_*R) \ar[r]^{H_{\mm}^i(\phi)}  & H_{\mm}^i(R) \\
L \ar[r] \ar@{^{(}->}[u]  & F^e_*R \otimes L \ar[r] \ar[u]   & F^e_*L \ar@{^{(}->}[u]\\
N \ar[r] \ar@{^{(}->}[u] & F^e_*R \otimes N  \ar@{->>}[r] \ar[u] & F^e_*N  \ar@{^{(}->}[u]
}\]
where the top part is commutative by \cite[Lemma 3.3]{Ma14} (see also \cite[Lemma 1.23 on page 493]{SchwedeSmithBook}), and the bottom right surjectivity follows from the hypothesis that $N$ is full. 

We need to show $\phi(F^e_*J)\subseteq J$. Since $J=\Ann_R(L/N)$, it suffices to show that $\phi(F^e_*J)L\subseteq N$. But by the commutative diagram above, we have 
$$\phi(F^e_*J)L\subseteq H_{\mm}^i(\phi)(F^e_*J\cdot F^e_*L)\subseteq H_{\mm}^i(\phi)(F^e_*N)=H_{\mm}^i(\phi)(\Image(F^e_*R \otimes N))\subseteq \phi(F^e_*R)N\subseteq N.$$
Hence the first statement follows. The second statement follows from the first by applying it to the full submodule $N=0$.
\end{proof}

\begin{proposition}
\label{p:radcomp}
Let $(R,\mm)$ be an $F$-finite complete local ring of prime characteristic $p$. If the annihilator $J$ of an $F$-stable subquotient of $H_{\mm}^i(R)$ is radical (e.g., if the Frobenius action on that subquotient is injective), then $J$ is a uniformly compatible ideal.

In particular, if $R$ is $F$-anti-nilpotent, then the annihilator of any $F$-stable subquotient of $H_{\mm}^i(R)$ is uniformly compatible.
\end{proposition}
\begin{proof}
Let $L/N$ be an $F$-stable subquotient of $H_\mm^i(R)$ so that $J:=\Ann(L/N)$ is radical. Suppose $J$ is not uniformly compatible. Then there is $\phi\in \Hom_R(F^e_*R, R)$ so that $\phi(F^e_*J)\nsubseteq J$. Since $J$ is radical, there exists a minimal prime $P$ of $J$ so that $\phi(F^e_*J)\nsubseteq P$. Thus after localizing at $P$, $\phi(F^e_*J_P)=R_P$. It follows that $R_P$ is F-pure and hence F-anti-nilpotent. 

The inclusion of $F$-stable submodules $N\subseteq L \subseteq H_\mm^i(R)$ yields an inclusion of Cartier submodules $C\subseteq D \subseteq H^{-i}(\omega_R^\bullet)$, with $\Ann_R(D/C)=J$. After localizing at $P$, we obtain an inclusion of Cartier submodules $C_P\subseteq D_P\subseteq H^{-j}(\omega_{R_P}^\bullet)$ with $J_P=\Ann_{R_P}(D_P/C_P)$, which in turn yields an inclusion of $F$-stable submodules $N'\subseteq L' \subseteq H_P^j(R_P)$ with $\Ann_{R_P}(L'/N')=J_P$. Since $R_P$ is $F$-anti-nilpotent, every $F$-stable submodule of $H_P^j(R_P)$ is full by \cite[Lemma 2.1]{MaQuyDeformation} and thus by Proposition~\ref{prop: F-full submodules}, $J_P$ is uniformly compatible in $R_P$, which contradicts $\phi(F^e_*J_P)=R_P$.
\end{proof}

Based on Proposition~\ref{prop: F-full submodules} and Proposition~\ref{p:radcomp}, one might ask whether the annihilator of every $F$-stable subquotient of $H_\mm^i(R)$ is uniformly compatible. We include a simple example indicating that this is not the case.

\begin{example}
Let $R=S/I$ with $S=\FF_2[[x,y]]$ and $I=(x,y)^2$, and let $\mm=(x,y)$. If we set $N=(x)$, then $N$ is an $F$-stable submodule of $L=R = H^0_\mm(R)$, and 
$$\Ann_R(L/N) = \Ann_R(R/(x)) = (x).$$ 
Now for $q=2^e$ with $e\geq 1$, we have $f_e:=x^{q-2}y^{2q-1} \in I^{[q]}:_SI$. Let ${\rm Tr}:F_*(S)\to S$ denote the trace map. We consider $\phi_e(-) = {\rm Tr}^e(f_e \cdot -):F^e_*(R) \to R$ and it is easy to see that 
$$\phi_e(x) = {\rm Tr}^e(x^{q-1}y^{2q-1}) = y{\rm Tr}^e(x^{q-1}y^{q-1}) = y \notin (x),$$ 
thus $(x)$ is not uniformly compatible.
\end{example}

\bibliographystyle{acm}
\bibliography{References}

\begin{thebibliography}{10}

\bibitem{BrionKumarFsplittingBook}
{\sc Brion, M., and Kumar, S.}
\newblock {\em Frobenius splitting methods in geometry and representation
  theory}, vol.~231 of {\em Progress in Mathematics}.
\newblock Birkh\"{a}user Boston, Inc., Boston, MA, 2005.

\bibitem{BH1993}
{\sc Bruns, W., and Herzog, J.}
\newblock {\em Cohen-{M}acaulay rings}, vol.~39 of {\em Cambridge Studies in
  Advanced Mathematics}.
\newblock Cambridge University Press, Cambridge, 1993.

\bibitem{CoVa}
{\sc Conca, A., and Varbaro, M.}
\newblock Square-free {G}r\"obner degenerations.
\newblock {\em Invent. Math. 221}, 3 (2020), 713--730.

\bibitem{DaoMaVarbaro}
{\sc Dao, H., Ma, L., and Varbaro, M.}
\newblock Regularity, singularities and {$h$}-vector of graded algebras.
\newblock {\em Trans. Amer. Math. Soc. 377}, 3 (2024), 2149--2167.

\bibitem{DaMuF-inj}
{\sc Datta, R., and Murayama, T.}
\newblock Permanence properties of {$F$}-injectivity.
\newblock {\em Math. Res. Lett. 31}, 4 (2024), 985--1027.

\bibitem{DSGNB}
{\sc De~Stefani, A., Grifo, E., and N\'u\~nez Betancourt, L.}
\newblock Local cohomology and {L}yubeznik numbers of {$F$}-pure rings.
\newblock {\em J. Algebra 571\/} (2021), 316--338.

\bibitem{EisenbudBook}
{\sc Eisenbud, D.}
\newblock {\em Commutative algebra}, vol.~150 of {\em Graduate Texts in
  Mathematics}.
\newblock Springer-Verlag, New York, 1995.
\newblock With a view toward algebraic geometry.

\bibitem{ElkiesSupersingular}
{\sc Elkies, N.~D.}
\newblock The existence of infinitely many supersingular primes for every
  elliptic curve over {${\bf Q}$}.
\newblock {\em Invent. Math. 89}, 3 (1987), 561--567.

\bibitem{Fed83}
{\sc Fedder, R.}
\newblock {$F$}-purity and rational singularity.
\newblock {\em Trans. Amer. Math. Soc. 278}, 2 (1983), 461--480.

\bibitem{G-M}
{\sc Gonz\'alez-Mart\'inez, R.}
\newblock Gorenstein binomial edge ideals.
\newblock {\em Math. Nachr. 294}, 10 (2021), 1889--1898.

\bibitem{CayleySalmon}
{\sc Hahn, M.~A., Lamboglia, S., and Vargas, A.}
\newblock A short note on {C}ayley-{S}almon equations.
\newblock {\em Matematiche (Catania) 75}, 2 (2020), 559--574.

\bibitem{HaraRationalSingularities}
{\sc Hara, N.}
\newblock A characterization of rational singularities in terms of injectivity
  of {F}robenius maps.
\newblock {\em Amer. J. Math. 120}, 5 (1998), 981--996.

\bibitem{HTVW}
{\sc {Huang}, H., {Tarasova}, Y., {Varbaro}, M., and {Witt}, E.}
\newblock {Smooth Herzog projective curves}.
\newblock {\em arXiv e-prints\/} (Nov. 2025), arXiv:2512.00584.

\bibitem{KoVa}
{\sc Koley, M., and Varbaro, M.}
\newblock Gr\"{o}bner deformations and {$F$}-singularities.
\newblock {\em Math. Nachr. 296}, 7 (2023), 2903--2917.

\bibitem{KollarMori}
{\sc Koll\'{a}r, J., and Mori, S.}
\newblock {\em Birational geometry of algebraic varieties}, vol.~134 of {\em
  Cambridge Tracts in Mathematics}.
\newblock Cambridge University Press, Cambridge, 1998.
\newblock With the collaboration of C. H. Clemens and A. Corti, Translated from
  the 1998 Japanese original.

\bibitem{KSSW}
{\sc Kurano, K., Sato, E.-i., Singh, A.~K., and Watanabe, K.-i.}
\newblock Multigraded rings, diagonal subalgebras, and rational singularities.
\newblock {\em J. Algebra 322}, 9 (2009), 3248--3267.

\bibitem{LyubeznikFmodules}
{\sc Lyubeznik, G.}
\newblock {$F$}-modules: applications to local cohomology and {$D$}-modules in
  characteristic {$p>0$}.
\newblock {\em J. Reine Angew. Math. 491\/} (1997), 65--130.

\bibitem{LSW}
{\sc Lyubeznik, G., Singh, A.~K., and Walther, U.}
\newblock Local cohomology modules supported at determinantal ideals.
\newblock {\em J. Eur. Math. Soc. (JEMS) 18}, 11 (2016), 2545--2578.

\bibitem{Ma14}
{\sc Ma, L.}
\newblock Finiteness properties of local cohomology for {$F$}-pure local rings.
\newblock {\em Int. Math. Res. Not. IMRN}, 20 (2014), 5489--5509.

\bibitem{MaPolstraFsingularitesBook}
{\sc Ma, L., and Polstra, T.}
\newblock F-singularities: a commutative algebra approach.
\newblock available at
  \url{https://www.math.purdue.edu/~ma326/F-singularitiesBook.pdf}.

\bibitem{MaQuyDeformation}
{\sc Ma, L., and Quy, P.~H.}
\newblock Frobenius actions on local cohomology modules and deformation.
\newblock {\em Nagoya Math. J. 232\/} (2018), 55--75.

\bibitem{MaSchwedeShimomoto}
{\sc Ma, L., Schwede, K., and Shimomoto, K.}
\newblock Local cohomology of {D}u {B}ois singularities and applications to
  families.
\newblock {\em Compos. Math. 153}, 10 (2017), 2147--2170.

\bibitem{QuyShimomoto}
{\sc Quy, P.~H., and Shimomoto, K.}
\newblock {$F$}-injectivity and {F}robenius closure of ideals in {N}oetherian
  rings of characteristic {$p>0$}.
\newblock {\em Adv. Math. 313\/} (2017), 127--166.

\bibitem{SakamakiCubicSurfaces}
{\sc Sakamaki, Y.}
\newblock Automorphism groups on normal singular cubic surfaces with no
  parameters.
\newblock {\em Trans. Amer. Math. Soc. 362}, 5 (2010), 2641--2666.

\bibitem{SchwedeSmithBook}
{\sc Schwede, K., and Smith, K.}
\newblock {\em Singularities defined by the {F}robenius map}.
\newblock 2024.

\bibitem{Anurag99}
{\sc Singh, A.~K.}
\newblock Deformation of {$F$}-purity and {$F$}-regularity.
\newblock {\em J. Pure Appl. Algebra 140}, 2 (1999), 137--148.

\bibitem{WatanabeRationalSingularities}
{\sc Watanabe, K.}
\newblock Rational singularities with {$k^{\ast} $}-action.
\newblock In {\em Commutative algebra ({T}rento, 1981)}, Heidelberger
  Taschenb\"{u}cher. Dekker, New York, 1983, pp.~339--351.

\end{thebibliography}

\end{document}